\documentclass[12pt,reqno,oneside]{amsart}
\usepackage{amsthm,amsmath,amssymb,bbm,hyperref,fullpage}

\theoremstyle{plain}
\newtheorem{thm}{Theorem}
\newtheorem{lem}[thm]{Lemma}
\newtheorem{cor}[thm]{Corollary}

\theoremstyle{definition}

\newcommand\Z{\mathbb{Z}}
\newcommand\eps{\varepsilon}
\newcommand\cP{\mathcal{P}}
\newcommand\comp{^{\mathrm{c}}}

\renewcommand{\le}{\leqslant}
\renewcommand{\ge}{\geqslant}

\begin{document}

\title{A note on Linnik's Theorem on quadratic non-residues}

\author[P. Balister]{Paul Balister}
\address{Department of Mathematical Sciences, University of Memphis, Memphis TN 38152, USA}
\email{pbalistr@memphis.edu}

\author[B. Bollob\'as]{B\'ela Bollob\'as}
\address{Department of Pure Mathematics and Mathematical Statistics, University of Cambridge,
Wilberforce Road, Cambridge CB3\thinspace0WB, UK; \emph{and} Department of Mathematical Sciences, University of Memphis, Memphis TN 38152, USA; \emph{and} London Institute for Mathematical Sciences, 35a South St., Mayfair, London W1K\thinspace2XF, UK}
\email{b.bollobas@dpmms.cam.ac.uk}

\author[J.D. Lee]{Jonathan D. Lee}
\address{Microsoft Research, Redmond, USA}
\email{jonatlee@microsoft.com}

\author[R. Morris]{Robert Morris}
\address{IMPA, Estrada Dona Castorina 110, Jardim Bot\^anico, Rio de Janeiro, RJ, Brazil}
\email{rob@impa.br}

\author[O. Riordan]{Oliver Riordan}
\address{Mathematical Institute, University of Oxford, Radcliffe Observatory Quarter,
Woodstock Road, Oxford OX2\thinspace6GG, UK}
\email{riordan@maths.ox.ac.uk}

\maketitle

\begin{abstract}
We present a short, self-contained and purely combinatorial proof of Linnik's theorem:
for any $\eps>0$ there exists a constant $C_\eps$ such that for any $N$, there are at most $C_\eps$
primes $p\le N$ such that the least positive quadratic non-residue modulo $p$ exceeds $N^\eps$.
\end{abstract}

\section{Introduction}
In 1941, Linnik~\cite{linniksieve} developed the Large Sieve and used it (in~\cite{linniknqr}) to
prove that for any $\eps>0$ there is a constant $C_\eps$ such that for all~$N$, there are at most
$C_\eps$ primes $p\le N$ such that the least quadratic non-residue modulo $p$ exceeds~$N^\eps$.
The Large Sieve was subsequently developed as a theorem about
quasi-independent functions by R\'enyi~\cite{renyi1,renyi2,renyi3,renyi4}, and as a theorem on
duality and an approximate Plancherel identity by Bombieri~\cite{bombieri} and Roth~\cite{roth}.
Modern presentations include~\cite{CMR2006,cribo}.

The techniques of these proofs are ultimately analytic, and relate Fourier analysis on the circle
to the discrete Fourier transform on $\Z/p\Z$. In this paper, we present a direct
combinatorial (or `elementary') proof of Linnik's theorem, based on an explicit combinatorial sieve.
To emphasize the elementary nature of this proof, we also include proofs of some (very weak)
estimates on sums of the reciprocals of the primes and smooth numbers.
In doing so, we avoid using standard results such
as the prime number theorem, which despite having an elementary proof (see~\cite{PNT}), is in
fact far deeper than the result we wish to show.

\section{Estimate on prime numbers}

The following is a very weak form of Mertens' Theorem~\cite{Mertens}. For completeness
we include an elementary proof, inspired by a proof given by Erd\H{o}s.

\begin{lem}\label{sum1overp}
For any fixed\/ $\eps\in(0,1)$, we have
\[
\sum_{\substack{n^{1-\eps} < \, p \, \le\, n\\p\text{ prime}}}\frac{1}{p} \, \ge \, \eps + o(1)
\]
as $n \to \infty$. 
\end{lem}

\begin{proof}
The form of the statement allows us to assume that $n$ is large, so replacing $n$ by $\lfloor n\rfloor$ we may assume that $n$ is an integer.

First note the elementary result that a prime $p$ divides $n!$ exactly $\sum_{k\ge 1}\lfloor\frac{n}{p^k}\rfloor<\frac{n}{p-1}$ times. Thus if $a_1+\dots+a_t=n$,
then the number of times $p$ divides the multinomial coefficient $n!/(a_1!\dots a_t!)$ is at most $\frac{n}{p-1}$, but is also
\begin{align*}
 \sum_{k \ge 1} \bigg( \Big\lfloor\frac{n}{p^k}\Big\rfloor - \sum_{i=1}^t\Big\lfloor\frac{a_i}{p^k}\Big\rfloor \bigg)
 & \, \le \, \sum_{k\ge 1,\,p^k\le n} \bigg( \frac{n}{p^k} - \sum_{i=1}^t \left( \frac{a_i}{p^k}-1 \right) \bigg)\\
 & \, = \, \sum_{k\ge1,\,p^k\le n} t \, = \, O\big( t \log n \big).
\end{align*}
However, the multinomial expansion of
\[
 t^n=(1+\dots+1)^n=\sum_{a_1+\dots+a_t=n}\frac{n!}{a_1!\dots a_t!},
\]
contains at most $(n+1)^t$ terms, as each $a_i$ satisfies $0\le a_i\le n$. Thus if we choose
$a_1,\dots,a_t$ so as to maximize the multinomial coefficient, we have $n!/(a_1!\dots a_t!)\ge t^n/(n+1)^t$.
Thus by considering the prime factorization of $n!/(a_1!\dots a_t!)$ we have, for any $r<n$,
\begin{align*}
 n\log t & \, \le \, \log\frac{n!}{a_1!\dots a_t!}+O(t\log n)\\
 & \, \le \, \sum_{\substack{r<p\le n\\p\text{ prime}}}\frac{n}{p-1}\log p
 +\sum_{\substack{p\le r\\p\text{ prime}}}O(t\log n)\log p+O(t\log n)\\
 & \, \le \, \sum_{\substack{r<p\le n\\p\text{ prime}}}\frac{n\log n}{p-1}+O\big( rt (\log n)^2 \big).
\end{align*}
Dividing by $n\log n$ gives
\[
 \sum_{\substack{r<p\le n\\p\text{ prime}}}\frac{1}{p-1}
 \ge\frac{\log t}{\log n}-O\bigg( \frac{rt\log n}{n} \bigg).
\]
Let $r=n^{1-\eps}$ and $t=n^\eps/(\log n)^2$. Then, noting that
\[
 \frac{1}{r}=\sum_{i>r}\bigg( \frac{1}{i-1} - \frac{1}{i} \bigg) \ge
 \sum_{\substack{r<p\le n\\p\text{ prime}}}\bigg( \frac{1}{p-1} - \frac{1}{p} \bigg),
\]
we have
\begin{align*}
 \sum_{\substack{n^{1-\eps}<p\le n\\p\text{ prime}}}\frac{1}{p}
 & \, \ge \, \bigg(\sum_{\substack{n^{1-\eps}<p\le n\\p\text{ prime}}}\frac{1}{p-1}\bigg)-\frac{1}{n^{1-\eps}}\\
 & \, \ge \, \frac{\eps\log n-2\log\log n}{\log n} - \frac{O(1)}{\log n} - \frac{1}{n^{1-\eps}}\\
 & \, \ge \, \eps - o(1)
\end{align*}
as $n \to \infty$, as required.
\end{proof}

\section{Estimate on smooth numbers}

We recall that a positive integer $m$ is called $y$\emph{-smooth} if for any prime $p$ dividing~$m$, $p\le y$. One of the earliest results on the number $\Psi(n,y)$ of $y$-smooth numbers in the set $\{1,\dots,\lfloor n\rfloor\}$ was obtained by Dickman~\cite{Dickman} in 1930. We will need the following weak version of his theorem; once again, we include an elementary proof for completeness.

\begin{lem}\label{psi-1/u}
For any $u > 0$, there exists a constant $c_u>0$ such that
\[
 \Psi\big( n, n^{1/u} \big) \ge c_u n
\]
for all real\/ $n\ge1$.
\end{lem}

\begin{proof}
We induct on $\lfloor 2u\rfloor$.
If $u \le 1$, then $n^{1/u}\ge n$ and $\Psi(n,n^{1/u})=\lfloor n\rfloor$, so we can take $c_u=1/2$.
Suppose $u>1$. By decomposing the set of $y$-smooth numbers by their largest
prime factor we have
\begin{equation}\label{eq:smooth:decomp}
 \Psi(n,y)=1+\sum_{\substack{p \, \le \, y\\p\text{ prime}}}\Psi\big(\tfrac{n}{p},p\big).
\end{equation}
Let $u' = u+\frac{1}{2}$, and note that $\log(n/p) / \log p \le u'-1 = u - \frac{1}{2}$ for any prime $p$ in the range $n^{1/u'} < p \le n^{1/u}$. Hence, by the induction hypothesis, we have $\Psi(\tfrac{n}{p},p)\ge c_{u-1/2} \, \frac{n}{p}$ for each such $p$, and therefore, by~\eqref{eq:smooth:decomp}, 
\[
 \Psi\big( n, n^{1/u} \big) \ge \sum_{\substack{n^{1/u'} < \, p \, \le \, n^{1/u}\\p\text{ prime}}}\Psi\big(\tfrac{n}{p},p\big) \, \ge \, c_{u-1/2} \sum_{\substack{n^{1/u'} < \, p \, \le \, n^{1/u}\\p\text{ prime}}} \frac{n}{p}.
\]
Now, $n^{1/u'}=(n^{1/u})^{1-\eps}$ with $\eps=1-\frac{u}{u'}=\frac{1}{2u'}$, and therefore, by Lemma~\ref{sum1overp},
\[
 \Psi\big( n, n^{1/u} \big) \, \ge \, c_{u-1/2} \sum_{\substack{n^{1/u'} < \, p \, \le \, n^{1/u}\\p\text{ prime}}} \frac{n}{p} \, \ge \, c_{u-1/2} \bigg( \frac{1}{2u'} - o(1) \bigg) n
\]
as $n \to \infty$. Since $\Psi(n,n^{1/u})\ge1$ for all $n\ge1$, it follows that there exists a $c_u>0$ such that $\Psi(n,n^{1/u})\ge c_u n$ for all $n\ge1$, as claimed.
\end{proof}

\section{The combinatorial sieve}

The following simple combinatorial inequality is in fact the key ingredient of the proof.

\begin{lem}\label{comb-sieve-2}
For any sets $A_1,\dots,A_d\subseteq [n]=\{1,2,\ldots,n\}$,
\[
 (d+1)^2\Big|\bigcap_{i=1}^{d}A_i\Big|\le (d+1)^2n-4d\sum_{i=1}^{d}|A\comp_i|+4\sum_{i\ne j}|A\comp_i\cap A\comp_j|.
\]
\end{lem}
\begin{proof}
Suppose $x\in [n]$ is in $\ell$ of the sets $A\comp_i$. Then on the right hand side it is counted
\[
 (d+1)^2-4d\ell+4\ell(\ell-1)=(d+1-2\ell)^2\ge0
\]
times and, if it is in all of the sets $A_i$, it is counted exactly $(d+1)^2$ times.
\end{proof}

\section{Linnik's Theorem}

For any prime number $p$, let $n_p$ be the least positive quadratic non-residue modulo~$p$.

\begin{thm}\label{cheap-sieve-bound}
Fix integers\/ $N>B\ge 2$. Let\/ $\cP=\{p \le N : p$ is prime and\/ $n_p>B\}$ and\/ $d=|\cP|$.
Then\/ $(d+1)\Psi(N^3,B)\le (5+dB^{-2})N^3$.
\end{thm}

\begin{proof}
Fix $n$ (to be chosen later) and,
for each $p_i\in\cP$, let $A_i= \big\{x\in[n]:\big(\frac{x}{p_i}\big)=1\}$. Each $A_i$ contains
all primes less than or equal to~$B$.
Furthermore, if $x,y\in A_i$ and $xy\le n$, then $xy\in A_i$. Hence,
\[
 \Big|\bigcap_{i=1}^d A_i\Big|\ge\Psi(n,B).
\]
The set $A\comp_i$ contains all integers up to $n$ in $(p_i+1)/2$ of the
congruence classes modulo~$p_i$, so
\[
 |A_i\comp|\ge \frac{n}{2} \bigg( 1 + \frac{1}{p_i} \bigg) - p_i,
\]
and hence
\[
 4d\sum_{i=1}^d|A_i\comp|\ge 2d^2n+2dn\sum_{p\in\cP} \frac{1}{p}-4d\sum_{p\in\cP}p.
\]
By the Chinese remainder theorem, when $i\ne j$ the set $A\comp_i\cap A\comp_j$ contains all integers
up to $n$ in $(p_i+1)(p_j+1)/4$ of the congruence classes modulo $p_ip_j$. Thus
\[
 |A_i\comp\cap A_j\comp| \le \frac{n}{4} \bigg( 1 + \frac{1}{p_i} \bigg) \bigg( 1 + \frac{1}{p_j} \bigg) + p_ip_j,
\]
and hence
\begin{align*}
 4\sum_{i\ne j}|A_i\comp\cap A_j\comp|&\le d(d-1)n+2(d-1)n\sum_{p\in\cP} \frac{1}{p} + n\!\!\!\sum_{p,q\in\cP,\,p\ne q} \frac{1}{pq}+4\!\!\sum_{p,q\in\cP,\,p\ne q}pq.
\end{align*}
Hence by Lemma~\ref{comb-sieve-2},
\begin{align*}
 (d+1)^2\Psi(n,B)
 &\le(d+1)^2n-2d^2n-2dn\sum_{p\in\cP}\frac{1}{p}+4d\sum_{p\in\cP}p\\
 &\qquad+d(d-1)n+2(d-1)n\sum_{p\in\cP}\frac{1}{p} + n\!\!\!\sum_{p,q\in\cP,\,p\ne q}\frac{1}{pq}+4\!\!\sum_{p,q\in\cP,\,p\ne q}pq\\
 &\le(d+1)n+4d\sum_{p\in\cP}p+n\!\!\!\sum_{p,q\in\cP,\,p\ne q}\frac{1}{pq}+4\!\!\sum_{p,q\in\cP,\,p\ne q}pq
\end{align*}
Note that $n_p<p$, and so each $p\in\cP$ is at least $B$ and at most~$N$. Hence
\[(d+1)^2\Psi(n,B)\le (d+1)n+4d^2N+\frac{nd^2}{B^2}+4d^2N^2.
\]
We now fix $n=N^3$. Then, as $d\le N$,
\[
 (d+1)^2\Psi(N^3,B)\le (d+1)N^3+4N^3+d^2B^{-2}N^3+4dN^3,
\]
and so $(d+1)\Psi(N^3,B)\le(5+dB^{-2})N^3$ as required.
\end{proof}

\begin{cor}
Fix $\eps>0$. Then there is a constant $C_\eps$ such that for all $N$ there are at most $C_\eps$
primes $p\le N$ with $n_p > N^\eps$.
\end{cor}
\begin{proof}
Recall from Lemma~\ref{psi-1/u} that $\Psi(n,n^{1/u})\ge c_u n$. Take $u = 3/\eps$ and $n=N^3$ so that
$\Psi(N^3,N^\eps)\ge c_{3/\eps} N^3$. Hence, by Theorem~\ref{cheap-sieve-bound},
the number $d$ of primes $p\le N$ with $n_p > B = N^\eps$ satisfies
\[
 d\left( \Psi(N^3,N^\eps)-\frac{N^3}{N^{2\eps}} \right) \, \le \, 5N^3-\Psi(N^3,N^\eps).
\]
Thus
\[
 d \, \le \, \frac{5N^3}{\Psi(N^3,N^\eps)-N^{3-2\eps}} \, \le \, \frac{5}{c_{3/\eps}-N^{-2\eps}},
\]
which is bounded as $N\to\infty$.
\end{proof}

\section*{Acknowledgement}
The work of the first two authors was partially supported by NSF grant DMS 1600742, and
work of the second author was also partially supported by MULTIPLEX grant 317532.
The work of the fourth author was partially supported by CNPq (Proc.~303275/2013-8) and FAPERJ (Proc.~201.598/2014).
The research in this paper was carried out while the third, fourth and fifth 
authors were visiting the University of Memphis.


\begin{thebibliography}{99}

\bibitem{bombieri}
E.~Bombieri, On the large sieve, \emph{Mathematika}, \textbf{12} (1965), 201--225.

\bibitem{CMR2006}
A.C. Cojocaru and M.R. Murty, An introduction to sieve methods and their applications, London Mathematical Society Student Texts, \textbf{66}, Cambridge University Press, 2006.

\bibitem{Dickman} K.~Dickman,
On the frequency of numbers containing prime factors of a certain relative magnitude,
\sl Ark. Mat. Astron. Fys., \bf 22 \rm (1930), 1--14.

\bibitem{cribo}
J.B. Friedlander and H.~Iwaniec, Opera de cribro, American Mathematical Society Colloquium Publications, vol. 57, American Mathematical Society, Providence, RI, 2010.

\bibitem{PNT}
D. Goldfeld, The elementary proof of the prime number theorem: an historical perspective, 
In: Number Theory (Chudnovsky D., Chudnovsky G., Nathanson M., eds), pp. 179--192, Springer, New York, 2004. 

\bibitem{linniksieve}
U.V. Linnik, The large sieve, \emph{C. R. (Doklady) Acad. Sci. URSS (N.S.)}, \textbf{30} (1941), 292--294. 

\bibitem{linniknqr}
U.V. Linnik, A remark on the least quadratic non-residue, \emph{C. R. (Doklady)  Acad. Sci. URSS (N.S.)}, \textbf{36} (1942), 119--120. 

\bibitem{Mertens} F. Mertens, Ein Beitrag zur analytischen Zahlentheorie, \emph{J. Reine Angew. Math.}, \textbf{78} (1874), 46--62.

\bibitem{renyi1}
A. R\'enyi, On the representation of an even number as the sum of a single prime and a single almost-prime number, \emph{Doklady Akad. Nauk SSSR (N.S.)}, \textbf{56} (1947), 455--458. 

\bibitem{renyi2}
A. R\'enyi, Un nouveau th\'eor\`eme concernant les fonctions ind\'ependantes et ses applications \`a la th\'eorie des nombres, \emph{J. Math. Pures Appl.}, \textbf{28} (1949), 137--149.

\bibitem{renyi3}
A. R\'enyi, On the large sieve of Ju.V. Linnik, \emph{Compos. Math.}, \textbf{8} (1951), 68--75.

\bibitem{renyi4}
A.~R\'enyi, New version of the probabilistic generalization of the large sieve, \emph{Acta Math. Acad. Sci. Hung.}, \textbf{10} (1959), 217--226.

\bibitem{roth}
K.F. Roth, On the large sieves of Linnik and R\'enyi, \emph{Mathematika}, \textbf{12} (1965), 1--9.

\end{thebibliography}
\end{document}